\documentclass{amsart}

\usepackage{amssymb}
\usepackage{amscd}
\usepackage[all,cmtip]{xy}
\usepackage{hyperref}
\usepackage{graphicx}

\newcommand{\C}{\mathbb{C}}
\newcommand{\A}{\mathbb{A}}
\newcommand{\Ps}{\mathbb{P}}
\newcommand{\Q}{\mathbb{Q}}
\newcommand{\Z}{\mathbb{Z}}
\newcommand{\R}{\mathbb{R}}
\newcommand{\F}{\mathbb{F}}

\DeclareMathOperator{\Gal}{Gal}

\DeclareMathOperator{\ev}{ev}

\newcommand{\Mat}{\mathrm{Mat}}

\newcommand{\gen}[1]{\langle{#1}\rangle}
\newcommand{\pceil}[1]{\lceil{#1}\rceil}

\newcommand{\limplies}{\Longleftarrow}

\theoremstyle{plain}
\newtheorem{theorem}{Theorem}[section]
\newtheorem*{theorem-a}{Theorem A}
\newtheorem*{theorem-b}{Theorem B}
\newtheorem{corollary}[theorem]{Corollary}
\newtheorem{lemma}[theorem]{Lemma}
\newtheorem{proposition}[theorem]{Proposition}

\newtheorem*{acknowledgments}{Acknowledgments}

\theoremstyle{definition}
\newtheorem{definition}[theorem]{Definition}

\newtheorem{remark}[theorem]{Remark}

\newtheorem{example}[theorem]{Example}

\newtheorem{problem}{Problem}

\begin{document}

\title[Slopes of $\Z_p$-covers]
{On slopes of $L$-functions of $\Z_p$-covers over the projective line}

\author{Michiel Kosters}
\address{Department of Mathematics, University of California, Irvine, CA 92697-3875, USA}
\email{kosters@gmail.com}

\author{Hui June Zhu}
\address{Department of Mathematics, 
State University of New York at Buffalo, NY 14260-2900, USA}
\email{hjzhu@math.buffalo.edu}

\date{\today}

\maketitle

\begin{abstract}
Let $\mathcal{P}: \cdots \rightarrow C_2\rightarrow C_1\rightarrow {\mathbb P}^1$ be a $\Z_p$-cover of the projective line over a finite field of cardinality $q$ and characteristic $p$ which ramifies at exactly one rational point. We study the $q$-adic valuations of the reciprocal roots  in $\C_p$ of $L$-functions associated to characters of the Galois group of $\mathcal{P}$.
We show that for all covers $\mathcal{P}$ such that the genus of $C_n$ is a quadratic polynomial in $p^n$ for $n$ large, the valuations of these reciprocal roots are uniformly distributed in the interval $[0,1]$. Furthermore, we show that for a large class of such covers $\mathcal{P}$, the valuations of the reciprocal roots in fact form a finite union of arithmetic progressions. 
\end{abstract}

\tableofcontents

\section{Introduction}

Let $k=\F_q$ be a finite field of cardinality $q=p^a$ with $p$ prime.
Let $C_0=\Ps^1_k$ be the projective line over $k$. Let $\mathcal{P}: \cdots \to C_2 \to C_1 \to C_0$ be a $\Z_p$-cover of smooth projective geometrically irreducible curves over $k$. This means that $\Gal(C_i /C_0) \cong \Z/p^i\Z$. We identify $\Gal(\mathcal{P}/C_0) $ with $\Z_p$. In this paper we always assume that the tower is totally ramified at one rational point, and is unramified at other points. Without loss of generality, we assume this point to be $\infty$.
By $|\Ps^1_k|$ we denote the set of closed points of the projective line over $k$. For $x \in |\Ps^1_k| \setminus \{\infty\}=|\A^1_k|$ we denote by $\mathrm{Frob}(x) \in \Z_p$ its Frobenius element. 

Let $\overline{\Z}_p$ be the ring of integers of $\overline{\Q}_p$.
Let $\chi: \Z_p \to \C_p^*$ be a non-trivial character of conductor (order) $p^{m_{\chi}}$ for some $m_\chi\geq 1$. We consider the $L$-function
\begin{align*}
L(\chi,s)= \prod_{x \in |\A^1_k|} \frac{1}{1- \chi(\mathrm{Frob}(x)) s^{\deg(x)}} \in 1+s \overline{\Z}_p[[s]] \subset  \C_p [[s]].
\end{align*}
By the Weil Conjectures, $L(\chi,s)$ is a polynomial (\cite[Theorem A]{BOM}).
We can write the polynomial $L(\chi,s)$ as
\begin{align*}
L(\chi,s) = \prod_{i=1}^{\deg(L(\chi,s))} (1- \alpha_{\chi,i} s),
\end{align*}
where $\alpha_{\chi,i} \in \overline{\Z}_p$. We let $v_q$ be the valuation on $\C_p$ normalized by $v_q(q)=1$. By the Weil conjectures, one has $0 \leq v_q(\alpha_{\chi,i}) \leq 1$ and the leading term of $L(\chi,s)$ has $q$-adic valuation $\deg(L(\chi,s))/2$  (by the functional equation the $\alpha_{\chi,i}$ come in pairs which multiply to $q$; \cite[Theorem A]{BOM}). We would like to understand the distribution of the $v_q(\alpha_{\chi,i})$ in the interval $[0,1]$. The multiset of all $v_q(\alpha_{\chi,i})$  consists of precisely the {\em slopes} of the $q$-adic Newton polygon of the polynomial $L(\chi,s)$. Given two characters $\chi, \chi'$ with $m_{\chi}=m_{\chi'}$, the multiset of all $v_q(\alpha_{\chi,i})$ is equal to the multiset of all $v_q(\alpha_{\chi',i})$ by Galois theory. 

 
We say that $\mathcal{P}$ is \emph{genus stable}  if 
the genus $g(C_n)$ of $C_n$ is quadratic in 
$p^n$ for $n$ large enough.
In general $g(C_n)\geq h(p^n)$ for some quadratic polynomial $h(-)$ for $n$ large enough. In fact, for any function $h: \Z_{\geq 1} \to \R$, there exists a $\Z_p$-cover $\mathcal{P}$ with $g(C_n) \geq h(n)$ for $n \geq 1$. In other words, the genus of $C_n$ can grow arbitrarily fast. If $\mathcal{P}$ is not genus stable, not much is currently known on the behavior of the $v_q(\alpha_{\chi,i})$.
On the other hand, one can show that $\mathcal{P}$ is genus stable if and only if $\deg(L(\chi,s))$ is a linear polynomial in $p^{m_{\chi}}$ for large enough $m_{\chi}$ (see \cite[Proposition 5.5]{WAN8} for a proof and further discussions). 

We say that the tower $\mathcal{P}$ is \emph{slope uniform} if the following holds: For every interval $[a,b] \subseteq [0,1]$ one has
\begin{align*}
\lim_{m_{\chi} \to \infty}    \frac{\# \{ i = 1,2, \ldots, \deg(L(\chi,s)):\ v_q(\alpha_{\chi,i}) \in [a,b] \}}{\deg(L(\chi,s))} = b-a
\end{align*}
where the limit is over any finite non-trivial character $\chi$
of conductor $p^{m_\chi}$. 
Intuitively, slope uniformity means that the $q$-adic valuations of the $\alpha_{\chi,i}$ for a character $\chi$ with conductor $p^{m_{\chi}}$ approach a uniform distribution on $[0,1]$ when $m_{\chi}$ goes to infinity.
 Our first main result is the following, which follows from 
Theorem \ref{444}.

\begin{theorem-a} \label{sta}
Assume that $\mathcal{P}$ is genus stable. Then $\mathcal{P}$ is slope uniform.
\end{theorem-a}

We say that $\mathcal{P}$ is \emph{slope stable} if there exists an integer $m' \in \Z_{\geq 1}$ such that the following holds. Let $\chi_0$ be a character with $m_{\chi_0}=m'$. Then for every character $\chi$ with $m_{\chi} \geq m'$ the multiset $\{v_q(\alpha_{\chi,i}): i=1,2,\ldots,\deg(L(\chi,s)) \} $ is equal to the multiset
\begin{align*}
\bigcup_{i=0}^{p^{m_{\chi}-m'}-1} \left\{  \frac{i}{p^{m_{\chi}-m'}}, \frac{v_q(\alpha_{\chi_0,1})+i}{p^{m_{\chi}-m'}},  \frac{v_q(\alpha_{\chi_0,2})+i}{p^{m_{\chi}-m'}},  \ldots, \frac{v_q(\alpha_{\chi_0,\mathrm{deg}(L(\chi_0,s))}) +i}{p^{m_{\chi}-m'} }  \right\} - \{0\}.
\end{align*} 
In the literature (see \cite{WAN7} for example), {\em slope stable} is commonly described as {\em forming a finite union of arithmetic progressions}.
In the rest of the paper we write {\em $m'$-slope stable} when we want to emphasize $m'$.

We observe that slope stable implies genus stable. Indeed, if $\mathcal{P}$ is slope stable, 
then  for $m_{\chi} \geq m_{\chi_0}$ the degree $\deg(L(\chi,s))=(\deg L(\chi_0,s)+1)p^{m_\chi-m_{\chi_0}}-1$ is 
a linear polynomial in $p^{m_{\chi}}$, and as discussed above, this means that $\mathcal{P}$ is genus stable.

Any $\Z_p$-cover $\mathcal{P}$ can be explicitly given by 
$f_0,f_1,\ldots,\in\F_q[X]$ with $p \nmid \deg(f_i)$ 
using Witt-vector equations and Artin-Schreier-Witt theory (see Section \ref{ha4} for details).
The cover $\mathcal{P}$ is genus stable if and only if 
 $\delta=\delta(\mathcal{P})=\max\left\{ \frac{\deg(f_i)}{p^i}: i=0,1,\ldots\right\}$ 
exists (see Section \ref{ha4}). In that case, one has $\deg(f_i) \leq \delta p^i$ for all $i$. 
Our second major result is the following. 

\begin{theorem-b}  
Fix $\delta \in \Q_{>0}$. 
There is a constant $C=C(\delta,q) \in \R_{\geq 0}$ such that 
for any $\mathcal{P}$ genus stable with $\delta=\delta(\mathcal{P})$ and $\deg(f_i) \leq \delta p^i-C$ 
for all $i$ large enough,  $\mathcal{P}$ is slope stable. 
\end{theorem-b}

The proof of Theorem B can be found in Section \ref{slo}. Our $C$ in Theorem B is explicit and quite small (see Remark \ref{consta}). We do not know whether genus stability is equivalent to slope stability, in other words, whether we can take $C=0$ in Theorem B.

As an intermediate result, we also study how much the $L$-functions depend on the defining equations, the $f_i$, of the $\Z_p$-cover. We prove that many coefficients of the $f_i$ do not influence the Newton polygon of the various $L$-functions (Theorem \ref{om2}).

We will now discuss what new methods are used in this paper for studying $L$-functions of $\Z_p$-towers. 
First of all we recall a bit of the history of development in this area.
Liu and Wan (in \cite{WAN9}) introduce the $T$-adic $L$-function 
$L(T,s)$, that deforms the $p$-adic $L$-function $L(\chi,s)$.
Namely $L(\chi,s)$ is a specialization of $L(T,s)$ at $T=\chi(1)-1$.
The authors of \cite{WAN7} bounded the $T$-adic Newton polygon $L(T,s)$ 
by applying the lower and upper bound of $L(\chi,s)$ for $m_\chi=1$.
In \cite{LIX} X. Li improves \cite{WAN9} by taking into account both the $p$-adic and $T$-adic valuation of coefficients of $L(T,s)$.  
These techniques do not apply to the study of 
genus stable $\Z_p$-covers.

Instead working with $T$-adic and $p$-adic valuation, we
introduce a deformation of $L(T,s)$ (and subsequently of $L(\chi,s)$) in Section 4. We will call it the $\pi$-adic $L$-function $L(\pi,s)$ whose 
coefficients lie in a power series ring $R$ in infinitely many variables (see Section 4 and also Remark \ref{R:remark1}).
The new function $L(\pi,s)$ behaves much better than $L(T,s)$ since it incorporates special $p$-adic properties of elements of $p$-th power order in $\C_p^*$. Furthermore, we work with a topology on the base ring 
that reflects the structure of $L$-functions more intimately.

Finally we state below some open problems regarding $\Z_p$-covers of a curve over a finite field in characteristic $p$.

\begin{problem}
Is genus stability equivalent to slope stability?
\end{problem}

\begin{problem}
Let $\mathcal{P}$ be a non genus stable cover. Note that $\mathcal{P}$ is not slope stable. Is $\mathcal{P}$ slope uniform? It might be possible to apply our techniques when the genus of $C_n$ is bounded above by a quadratic polynomial in $p^n$ ({\em`almost genus stable'}). Almost genus stability is equivalent to saying that the set $\{\deg(f_i)/p^i: i=0,1,\ldots\}$ has a supremum. It seems that our theory has no handle on other cases.
\end{problem}

\begin{problem}
Set $K=k(X)$. There are various constructions in algebraic geometry, for example through \'etale cohomology, which give rise to continuous Galois representations
\begin{align*}
\Gal(K^{\mathrm{sep}}/K) \to \mathrm{GL}_1(\Z_p) \cong \Z_p^*.
\end{align*}
For $p=2$, one has $\Z_p^* \cong \Z/2\Z \times \Z_p$, and for $p>2$, one has $\Z_p^* \cong \F_p^* \times \Z_p$. 
Hence such representations give rise to continuous homomorphisms $\Gal(K^{\mathrm{sep}}/K) \to \Z_p$, which if non-trivial, correspond to $\Z_p$-covers of $\Ps^1_k$. Are these $\Z_p$-covers genus stable, slope uniform, or slope stable? It seems to be true in certain cases (see \cite{KRA}).
\end{problem}

\begin{problem}
Set $K=k(X)$. Let $\mathcal{P}$ be a $\Z_p$-cover given by $f \in W(K)$ as in Section \ref{ha4}. Then for $c \in \Z_{q}^*$, one can consider the `twisted' $\Z_p$-cover given by $cf$. Let $\chi$ be a non-trivial character of finite order. Are the $q$-adic Newton polygons of $L_f(\chi,s)$ and $L_{cf}(\chi,s)$ the same? 
A Witt-vector computation shows that this problem is related to the following problem. Consider the Witt-vector equation that defines 
the $\Z_{p^b}$-cover of ${\mathbb P}^1_k$:
$F^b y - y = f$ where $F$ is the Frobenius map and $b|a$. Does the $q$-adic Newton polygon of an $L$-function of a finite character $\chi: \Z_{p^b} \to \C_p^*$ of order $p^{m_{\chi}}>1$ only depend on $m_{\chi}$? Our result Theorem \ref{444} gives 
lower and upper bounds for the $q$-adic Newton polygons of both $L_f(\chi,s)$ and $L_{cf}(\chi,s)$. A similar result was first obtained by Ren, Wan, Xiao and Yu in \cite{REN}. However, it is unknown even for $\chi$ with $m_{\chi}=1$ if the $q$-adic Newton polygons of $L_f(\chi,s)$ and $L_{cf}(\chi,s)$ are the same.
\end{problem}

\begin{problem}
Find and implement an efficient algorithm for computing $L(\chi,s)$ or the $q$-adic Newton polygon of $L(\chi,s)$ for a $\Z_p$-cover and a finite non-trivial character $\chi$.
\end{problem}

\begin{problem}
Let $\mathcal{P}$ be a $\Z_p$-cover. What can one say about the complex numbers $\alpha_{\chi,i}$ when $m_{\chi} \to \infty$? One can write $\alpha_{\chi,i} = q^{1/2} e^{2 \pi i \theta_{\chi,i}}$  with $\theta_{\chi,i} \in \R/\Z$ for $i=1,\ldots,\deg(L(\chi,s))$. Is the multiset $\{\theta_{\chi,i}: i=1,\ldots, \deg(L(\chi,s))\}$ uniformly distributed in $\R/\Z$ when $m_{\chi} \to \infty$?
\end{problem}

\begin{acknowledgments}
We would like to thank Daqing Wan for introducing us to this topic, for his suggestions and for his interesting problems and conjectures. We would also like to thank the referee for his or her detailed comments and suggestions.
\end{acknowledgments}

\section{Artin-Hasse exponential}
Define the Artin-Hasse exponential by the formal power series
\begin{align*}
E(T) = \mathrm{exp} \left( \sum_{i=0}^{\infty} \frac{T^{p^i}}{p^i} \right) \in 1 + T + T^2 \Z_{(p)}[[T]]
\end{align*}
where $\Z_{(p)}$ is the localization of $\Z$ at the prime $p$. Let $S$ be either $\Z_{(p)}$, $\Z_p$, $\Z/p^n\Z$ or $\F_p$. We can view $E(T)$ as an element of $S[[T]]$. We let $v_T$ be the $T$-adic valuation on $S[[T]]$.

\begin{lemma} \label{saso}
For $i \in \Z_{\geq 1}$ the map
\begin{align*}
E(\cdot ): T^i S[[T]] \to 1 + T^i S[[T]]
\end{align*}
is a bijection, such that
$E(sT^i+ T^{i+1}S[[T]])=1 + s T^i + T^{i+1}S[[T]]$ for any $s\in S$.
In particular, $v_T(r)=v_T(E(r)-1)$ for any $r\in TS[[T]]$. 
\end{lemma}
\begin{proof}
One easily sees that the maps are defined.  Let $s \in 1+ T^i S[[T]]$. We show by induction that for $j \geq i$ there is a $h_j \in T^iS[[T]]$, unique modulo $T^j S[[T]]]$, such that $E(h_j)-s \in T^j S[[T]]$. For $i=j$, the statement holds. Assume that the statement holds for $j$. Then there is $h_j \in T^iS[[T]]$, unique modulo $T^j S[[T]]$, with $E(h_j)-s-cT^{j} \in T^{j+1}S[[T]]$ for some $c \in S$. Then one has $E(h_j-d T^{j})-s-(c-d)T^j \in T^{j+1}S[[T]]$ and hence only for $c=d$ one has $E(h_j-dT^j)-s \in T^{j+1}S[[T]]$. Hence we have to set $h_{j+1}=h_j-cT^j$, modulo $T^{j+1}S[[T]]$. The first result follows. The proof of the second result is easy.
\end{proof}

Define $\pi_i(T)\in T \Z_p[[T]]$ for $i \in \Z_{\geq 0}$ by Lemma \ref{saso} as follows:
\begin{align*}
\pi_i(T) = E^{-1} \left( (1+T)^{p^i} \right).
\end{align*}
We let $v_p$ be the valuation on $\C_p$ with $v_p(p)=1$.
Let $x \in \C_p$ with $v_p(x)>0$. Then for all $i \geq 0$ one has
$v_p(\pi_i(x))>0$. On the other hand, 
\begin{equation}\label{equal}
v_p(\pi_i(x))=v_p(E(\pi_i(x))-1)=v_p((1+x)^{p^i}-1).
\end{equation}
Below we shall study $v_p(\pi_i(x))$ via this equality.

\begin{lemma} \label{200}
Let $n \in \Z_{\geq 1}$. Let $x=\zeta_{p^n}-1 \in \overline{\Q}_p$ where $\zeta_{p^n}$ is a primitive $p^n$-th root of unity.  Then one has  
\begin{align*}
v_p(\pi_i(x))=p^i v_p(x)
\end{align*}
 for $0 \leq i<n$ and $\pi_i(x)=0$ for $i \geq n$.  
\end{lemma}
\begin{proof}
For $0\leq i<n$,  we have $v_p(\zeta_{p^n}^{p^i}-1)=p^i v_p(\zeta_{p^n}-1)$;
hence by (\ref{equal}) we have $v_p(\pi_i(x))=v_p(\zeta_{p^n}^{p^i}-1)=p^iv_p(x)$.
For $i \geq n$, $E(\pi_i(x))=\zeta_{p^n}^{p^i}=1$, hence $\pi_i(x)=0$ by (\ref{equal}).
\end{proof}

\begin{lemma}
\label{summer}
For $i \in \Z_{\geq 0}$ one has
\begin{align*}
\pi_i(T)\in T^{p^i} (1+T\Z_p[[T]])+ pT\Z_p[[T]].
\end{align*}
Furthermore, if we write $\pi_i(T) = \sum_{j} b_{ij} T^j \in \Z_p[[T]]$, then for any $j \in \Z_{\geq 0}$ one has
\begin{align*}
 \lim_{i \to \infty} b_{ij} = 0.
\end{align*}
\end{lemma}
\begin{proof}
The following diagram commutes, where the vertical maps are the natural mod $p$ projections:
\begin{align*}
\xymatrix{ 
T \Z_p[[T]]  \ar[r]^{E(\cdot)} \ar@{->>}[d]^{\tau_1} & 1+T \Z_p[[T]] \ar@{->>}[d]^{\tau_2} \\
T \F_p[[T]] \ar[r]^{E(\cdot)} & 1+T \F_p [[T]].
 }
\end{align*}
Hence we find 
\begin{align*}
E(\tau_1(\pi_i(T)) =\tau_2( E(\pi_i(T)))= \tau_2\left( (1+T)^{p^i}\right)=1+T^{p^i}.
\end{align*}
From Lemma \ref{saso} one obtains $\tau_1(\pi_i(T)) \in  T^{p^i} (1+T \F_p[[T]])$. The result follows. 

For the second result, note that the $p$-adic valuation of ${p^i}\choose{j}$ goes to infinity if $i$ goes to infinity and $j$ is fixed. The result then follows by studying the map $E: T\cdot \Z/p^n\Z [[T]] \to 1 + T\cdot \Z/p^n\Z [[T]]$ and the expansion of $(1+T)^{p^i}$, and applying Lemma \ref{saso}. 
\end{proof}

\section{Genus stable $\Z_p$-covers}  \label{ha4}

Let $k=\F_q$ be a finite field of cardinality $q=p^a$. By $\F_{q^i}$ we denote the extension of $\F_q$ of degree $i$ in a fixed algebraic closure of $\F_q$. We let $\Z_q$ be the ring of integers of the unramified extension of $\Q_p$ with residue field $\F_q$. Let $K=k(X)$, the function field of $\Ps^1_k$. We are interested in studying $\Z_p$-covers of $\Ps_k^1$. Such a cover is also called an Artin-Schreier-Witt cover. 
We  restrict to $\Z_p$-covers which ramify only at $\infty$, and which are totally ramified at $\infty$. By $W(K)$ we denote the $p$-typical Witt-vectors of $K$. See \cite{WAN8} and \cite{WAN7} for more background material of this section.
We use the symbol $[\; ]$ to represent the Teichm\"uller map. Such a tower, together with an isomorphism of the Galois group with $\Z_p$, can be given (almost uniquely) by 
\begin{align*}
f= c_0 + \sum_{i\geq 1, (i,p)=1} c_i[X]^i \in W(K)
\end{align*}
with $c_i \in \Z_q$ and $v_p(c_i) \to \infty$ as $i \to \infty$ and $\min\{v_p(c_i): (i,p)=1\}=0$ (see \cite[Proposition 4.3]{WAN8}). Let $\mathcal{P}: \ldots \to C_2 \to C_1 \to C_0 = \Ps^1_k$ be the corresponding $\Z_p$-cover. More concretely, in terms of function fields one has $k(C_i)=K(y_0,y_1,\ldots,y_{i-1})$ where $y=(y_0,y_1,\ldots) \in W( \overline{K})$ is such that $(y_0^p,y_1^p,\ldots)-(y_0,y_1,\ldots)=f \in W(K)$ (see \cite{WAN8}).  We identify the Galois group of $\mathcal{P}$ and $\Z_p$ by the following isomorphism:
\begin{align*}
\Gal(\mathcal{P}/C_0)=\Gal(K(y_0,y_1,\ldots)/K) \to& \Z_p \\
\sigma \mapsto& \sigma y - y.
\end{align*}
For $x \in |\A^1_k|$ one has (\cite[Lemma 5.8]{WAN8})
\begin{align*}
\mathrm{Frob}(x) = \mathrm{Tr}_{\Z_{q^{\deg(x)}}/\Z_p}\left(  c_0 + \sum_{i\geq 1, (i,p)=1} c_i[x']^i   \right) \in \Z_p,
\end{align*}
where $x' \in \overline{k}$ is any representative of $x$.

The conductor of $C_{m'} \to C_0$ from class field theory satisfies (\cite[Proposition 4.14]{WAN8})
\begin{align*}
\mathfrak{f}(C_{m'} \to C_0)= \max\left\{1 +  ip^{m_{\chi}-v_p(c_i)-1}: i\geq 1, (i,p)=1,\ v_p(c_{i})<m' \right\} \cdot \infty.
\end{align*}
Let $\chi: \Z_p \to \C_p^*$ be a non-trivial finite character. By the Weil Conjectures
 (see \cite[Theorem A]{BOM} and \cite[Proposition 4.14]{WAN8})
$L(\chi,s)$ is a polynomial of degree
\begin{align*}
\deg(L(\chi,s))=-1 +  p^{m_\chi-1}\max\left\{ \frac{i}{p^{v_p(c_i)}}: i\geq 1, (i,p)=1,\ v_p(c_i)<m_{\chi} \right\}.
\end{align*}
We can represent $f$ in another way. There are unique polynomials
\begin{align*}
f_i=\sum_j a_{ij} X^j \in \F_q[X]
\end{align*}
such that
\begin{eqnarray}\label{E:a_ij}
f= \sum_{i=0}^{\infty} p^i \sum_j [a_{ij}] [X]^j.
\end{eqnarray}
Write $d_i=\deg(f_i)$. One has $a_{ij}=0$ if $j\in p\Z_{\geq 1}$ and $\deg(f_0) \geq 1$ and $p \nmid d_i$.
The tower $\mathcal{P}$ is genus stable if and only if the set
$
\left\{ \frac{d_i}{p^i}: i=0,1,\ldots\ \right\}
$
has a maximum $\delta$, and in that case there is a unique $m$ with 
$$
\delta = \frac{d_m}{p^m} = \max \left\{ \frac{d_i}{p^i}: i=0,1,\ldots\ \right\} .
$$
See \cite[Proposition 5.5]{WAN8} for a proof. \\
Assume from now on that $\mathcal{P}$ is genus stable with the above notation. One has $\deg(f_i) \leq \delta p^i$. Furthermore, for $m' >m$ one has
$$
\mathfrak{f}(C_{m'}/C_0)= (1+d_m p^{m'-m-1})\cdot \infty
$$
and for a character $\chi$ with $m_{\chi}>m$ one has
\begin{align*}
\deg(L(\chi,s))= d_m p^{m_{\chi}-m-1}-1 = \delta p^{m_{\chi}-1}-1  .
\end{align*}

\textbf{For the rest of the paper we assume that we are given a genus stable 
$\Z_p$-cover $\mathcal{P}$ of ${\mathbb{P}}^1_{\F_q}$ with $\delta=d_m/p^m$ and coefficients 
$a_{ij}$, as discussed in this section.}

\section{Extrapolating $L$-functions} \label{s:4}

\subsection{The $\pi$-adic valuation} \label{543}

Suppose we are given a genus stable $\Z_p$-cover $\mathcal{P}$ with $\delta(\mathcal{P})=\delta$ 
and defining coefficients $a_{ij}$ where $(i,j)$ lies in 
\begin{align*}
\mathfrak{X}= \left\{(i,j) \in  \Z_{\geq 0}^2: j \leq \delta p^i=d_m p^{i-m},\ j \not \in p \Z_{\geq 1} \right\}  \subset \Z_{\geq 0}^2.
\end{align*}

Consider the formal power series ring	
\begin{align*}	
R=\Z_p[[\pi_{x}: x \in \mathfrak{X}]].		
\end{align*}
We will put a valuation on this ring, similar to the valuation on the ring $\Q_p[[X]]$ which sends $cX^i(1+ X\Q_p[[X]])$ to $i$. 
Consider the monoid	
\begin{align*} U = \left\{u: \mathfrak{X} \to \Z_{\geq 0}: u(x)=0 \textrm{ for almost all }x\right\}.
\end{align*}
For each $u \in U$ we set		
\begin{align*}		
\pi^u = \prod_{x \in \mathfrak{X}} \pi_x^{u(x)}.		
\end{align*}		
Any $r \in R$ can now be written in the form $r=\sum_{u \in U} a_u \pi^u$ where $a_u \in \Z_p$. For such $r \neq 0$ we set
\begin{align*}
v(r) =\min\left\{ \frac{1}{\delta}  \sum_{(i,j) \in \mathfrak{X}}  j \cdot u\left((i,j)\right): u \in U,\ a_u \neq 0 \right\} \in \frac{1}{\delta} \Z_{\geq 0}
\end{align*}	
We set $v(0)=\infty$. Note that $v(\pi_{(i,j)})=\frac{j}{\delta}$ and $v(r)=0$ for $r \in \Z_p \setminus \{0\}$.

\begin{proposition}
The ring $R$ is a domain. We can extend $v$ to the quotient field $Q(R)$ of $R$ by
\begin{align*}
v: Q(R) \to& \frac{1}{\delta} \Z \sqcup \{\infty\} \\		
c/b \mapsto& v(c)-v(b).
\end{align*}		
Then $(Q(R),v)$ is a discrete valued field, that is, for $r, s \in Q(R)$ one has		
\begin{align*}		
v(r+s) \geq \min\{v(r), v(s)\} \hspace{0.5cm} \textrm{and} \hspace{0.5cm} v(rs) = v(r)+v(s).
\end{align*}
Furthermore, let $R''=\{x \in Q(R): v(x) \geq 0\}$ be the discrete valuation ring of $v$. Then $R \subsetneq R''$ and $R$ is complete with respect to $v$.
\end{proposition}
\begin{proof}
Let $r,s \in R$. One can easily see that 
 $v(r+s) \geq \min\{v(r), v(s)\}$.
We will show $v(rs)=v(r)+v(s)$. Assume that $r, s \neq 0$. 
Notice that $v(-)$ combined with a fixed well order on $\mathfrak{X}$ leads to a weighted lexicographic monomial order on the set of monomials of $R$. Then the lowest
monomial in $rs$ is the product of the lowest monomial in $r$ and the lowest monomial in $s$. Hence $rs \neq 0$ and $v(rs)=v(r)+v(s)$. 

Extend the definition of $v$ to $Q(R)$ by 
$v(c/b)=v(c)-v(b)$. It is easy to check that $(Q(R),v)$ is a valued field.

By definition, we have $R \subseteq R''$. Note that $\frac{\pi_{(i,j)}}{\pi_{(i',j)}} \in R'' \setminus R$ if $i \neq i'$. Finally, one easily sees that $R$ is complete with respect to $v$.
\end{proof}

Note that $v(\pi_{(i,j)})=\frac{j}{\delta}$. 
For $(i,j) \in \mathfrak{X}$ we set
\begin{align*}
\Delta_{(i,j)} = p^i - v(\pi_{(i,j)}).
\end{align*}
One has
$\Delta_{(i,j)}=p^i - \frac{j}{\delta} \leq p^i$. We will now prove a lower bound on $\Delta_{(i,j)}$. 

\begin{lemma} \label{540}
Let $x \in \mathfrak{X}$. Then one has $\Delta_x \geq 0$ with $\Delta_x=0$ if and only if $x=(m,d_m)$. 
\end{lemma}
\begin{proof}
This follows easily from the definition of $\mathfrak{X}$. 
\end{proof}

Recall $\pi_i(T)$ are power series with $E(\pi_i(T))=(1+T)^{p^i}$. We will now define two ring morphisms. 

First we have the following (continuous) $\Z_p$-algebra morphism:
\begin{align*}
\ev_{T}: R \to& \Z_p[[T]]
\end{align*}
which sends $\pi_{(i,j)}$ to $\pi_i(T) \in \Z_p[[T]$. 

Secondly, let $\chi: \Z_p \to \C_p^*$ be a continuous group homomorphism. Set $\pi_\chi:=\chi(1)-1$ and assume $v_p(\pi_{\chi})>0$.  We then define the following (continuous) $\Z_p$-algebra morphism
\begin{align*}
\ev_{\chi}: R \to& \Z_p[[\pi_{\chi}]] \subseteq \C_p
\end{align*}
which sends $\pi_{i,j}$ to $\pi_i(\pi_{\chi}) \in \Z_p[[\pi_{\chi}]]$.

\subsection{The $\pi$-adic $L$-function}

For a given genus stable cover $\mathcal{P}$ over $\F_q$ 
with $\delta$ and coefficients $(a_{ij})_{(i,j)\in\mathfrak{X}}$,
we introduce our $\pi$-adic $L$- and $C$-functions
with coefficients in $R$, 
which are deformations of classical $p$-adic and $T$-adic  
$L$- and $C$-functions, respectively. Recall that $[~]$ denotes the Teichm\"uller map. 


For $k \in \Z_{\geq 1}$ we define the $\pi$-adic exponential sum:
\begin{align*}
S^*(k,\pi) = \sum_{x \in \F_{q^k}^*} 
\prod_{(i,j) \in \mathfrak{X}}  E(\pi_{(i,j)})^{\mathrm{Tr}_{\Z_{q^{k}}/\Z_p}( [a_{ij}] [x]^j )} \in R. 
\end{align*}

\begin{definition}
We define the $\pi$-adic $L$-functions of our $\Z_p$-cover by
\begin{align*}
L^*(\pi,s) =& \exp \left( \sum_{k=1}^{\infty} S^*(k,\pi) \frac{s^k}{k} \right) \\
=&  \prod_{x \in |\A^1_k|\setminus \{0\} } \frac{1}{1- s^{\deg(x)} \prod_{(i,j) \in \mathfrak{X}}   E(\pi_{(i,j)})^{\mathrm{Tr}_{\Z_{q^{\deg(x)}}/\Z_p}( [a_{ij}] [x]^j )} } \in 1+ s R[[s]]
\end{align*}
and
\begin{align*}
L(\pi,s)= \frac{L^*(\pi,s)}{1-   s  \prod_{(i,j) \in \mathfrak{X}: j=0}   E(\pi_{(i,j)})^{\mathrm{Tr}_{\Z_{q^{\deg(x)}}/\Z_p}( [a_{ij}] )}} \in 1 + s R[[s]].
\end{align*}
We define the $\pi$-adic characteristic function by
\begin{align*}
C^*(\pi,s) = \exp\left( \sum_{k=1}^{\infty} \frac{1}{1-q^k} S^*(k,\pi) \frac{s^k}{k} \right) \in 1+ sR[[s]].
\end{align*}
\end{definition}

One has
\begin{align*}
C^*(\pi,s)= L^*(\pi,s) L^*(\pi,qs) L^*(\pi, q^2s) \cdots, \textrm{ and } L^*(\pi,s) = \frac{C^*(\pi,s)}{C^*(\pi,qs)}. 
\end{align*}

We extend $\ev_T$ to a ring homomorphism  $R[[s]] \to \Z_p[[T]][[s]]$ by using $\ev_T$ coefficient wise. Similarly, if 
$\chi: \Z_p \to \C_p^*$ is a continuous character with $v_p(\pi_{\chi})>0$, then we extend $\ev_{\chi}$ to $\ev_{\chi}: R[[s]] \to \Z_p[[\pi_{\chi}]][[s]]$. 

The classical $L$-function is given by
\begin{align*}
L(\chi,s) = \prod_{x \in |\A^1_k|} \frac{1}{1- \chi(\mathrm{Frob}(x)) s^{\deg(x)}}= \ev_{\chi} (L(\pi,s)).
\end{align*}

\begin{remark} \label{jus}
Other specializations of $S^*$, $L^*$ and $C^*$ give rise to known functions in $\Z_p[[T]][[s]]$. The following functions are introduced in \cite{WAN7} to study $L(\chi,s)$:
\begin{align*}
S^*(k,T)=& \ev_T(S^*(k,\pi)), \\
L^*(T,s) =& \ev_T(L^*(\pi,s)), \\
C^*(T,s) =& \ev_T(C^*(\pi,s)).
\end{align*}
\end{remark}

\begin{remark}
\label{R:remark1}
We shall mainly focus on two types of Newton polygons in this paper.

For $h=\sum_{i}b_i s^i \in R[[s]]$ we define its $\pi$-adic Newton polygon to be the lower convex hull of $\{(i,v(b_i)): i=0,1,\ldots\}$. 

Secondly, let $\chi: \Z_p \to \C_p^*$ a finite non-trivial character. Recall that $v_{\pi_{\chi}}$ is the valuation on $\C_p$ with $v_{\pi_{\chi}}(\pi_{\chi})=1$. For $h=\sum_{i}b_i s^i  \in \C_p[[s]]$ we define its $\pi_{\chi}$-adic Newton polygon to be the lower convex hull of 
$\{ (i,v_{\pi_{\chi}}(b_i)): i=0,1,\ldots\}$.  
\end{remark}

\section{Dwork's theory}

We prove $\pi$-adic Hodge lower bound for the $C$- and $L$-functions defined in previous section.
For convenience of the reader, this section is largely self-contained.
Most arguments here are similar to that of \cite{WAN9}.

Retain $R=\Z_p[[\pi_{x}: x \in \mathfrak{X}]]$ with the $\pi$-adic 
valuation defined by $v(-)$ in Section \ref{s:4}. Let $R'=\Z_q[[\pi_{x}: x \in \mathfrak{X}]]$, with a similarly defined valuation $v$.
Let 
\begin{align*}
B=R'\gen{X}=
\left\{\sum_{i=0}^{\infty} c_i X^i\in R'[[X]]: v(c_i)\rightarrow \infty \mbox{ as } i\rightarrow \infty \right\}. 
\end{align*}
This $B$ is a 
$\pi$-adic Banach $R'$-module. The set
$\Gamma=\{X^i: i \in \Z_{\geq 0}\}$ is a basis of $B$ as Banach $R'$-module. 
For any integer $n\geq 1$, we write $\mu_n=\{x \in \C_p: x^n=1\}$.  

\begin{lemma}[Commutativity] \label{112}
Let $E(-)$ be the Artin-Hasse exponential function and let $T$ be a variable, 
then we have the following commutative diagram
\begin{align*}
    \xymatrix{ \mu_{q-1} \sqcup \{0\} \ar[d]_{\mathrm{Tr}_{\Z_q/\Z_p} } \ar[r]^{E(\cdot T)}  \ar[d] &  \Z_q[[T]] \ar[d] \ar[d]^{ \mathrm{Norm}_{\Z_q[[T]]/\Z_p[[T]]} } \\
\mu_{p-1} \sqcup \{0\} \ar[r]^{E(T)^{\cdot}}  & \Z_p[[T]].
 }
\end{align*}
\end{lemma}
\begin{proof}
For $x \in \mu_{q-1}$ one has $\sum_{j=0}^{a-1} x^{p^j}= \sum_{j=0}^{a-1} x^{p^{j+i}}$. Hence one has
\begin{align*}
E(T)^{\mathrm{Tr}_{\Z_q/\Z_p}(x)} =&  E(T)^{x+x^p+\ldots+x^{p^{a-1}}} = \mathrm{exp}\left( \sum_{i=0}^{\infty} \frac{T^{p^i}}{p^i} \sum_{j=0}^{a-1} x^{p^{j+i}} \right)\\
=&  E(T x)E(T x^p)\cdots E(T x^{p^{a-1}}) = \mathrm{Norm}_{\Z_q[[T]]/\Z_p[[T]]}(E(Tx)).
\end{align*}
\end{proof}

Recall our genus stable $\Z_p$-cover $\mathcal{P}$ 
with its defining data: for each $(i,j)\in \mathfrak{X}$,
$f_i=\sum_{0\leq j\leq \delta p^i}a_{ij}X^j$ for $a_{ij}\in k$ and $\delta=\frac{d_m}{p^m}$. Recall that $[a_{ij}]\in \Z_q$ is the Teichm\"uller lift of $a_{ij}$. 
For $(i,j) \in \mathfrak{X}$ set
\begin{align*}
E_{f_{(i.j)}} = E(  \pi_{(i,j)} [a_{ij}] X^j ) \in R'[[X]] ,\quad E_f(X) = \prod_{x \in \mathfrak{X}}  E_{f_x} \in R'[[X]].
\end{align*}
Then we can write 
$
E_f(X) =\sum_{k=0}^{\infty} \eta_k X^k
$
with $\eta_k \in R'$. 

\begin{lemma} \label{104}
For any $k \in \Z_{\geq 0}$, one has $v(\eta_k) \in \{ \frac{k}{\delta}, \infty\}$.
In particular $E_f(X)\in B$.
\end{lemma}
\begin{proof}
The first statement follows from definition $v(\pi_{(i,k)} ) = \frac{k}{\delta}$ and direct computation: the reason is that $\pi_{(i,j)}$ is always coupled with $X^j$ in the definition of $E_{f_{(i,j)}}$. The second statement follows from the first.
\end{proof}

\begin{lemma}[Dwork's splitting lemma] \label{120}
For any $x \in \F_{q^k}$ one has
$$
\prod_{(i,j) \in \mathfrak{X}}  E(\pi_{(i,j)})^{\mathrm{Tr}_{\Z_{q^{k}}/\Z_p}( [a_{ij}] [x]^j )}   = \prod_{l=0}^{ak-1} E_f^{\sigma^l}([x]^{p^l}) \in R'
$$
\end{lemma}
\begin{proof}
For $(i,j) \in \mathfrak{X}$ on has by Lemma \ref{112}:
\begin{align*}
E(\pi_{(i,j)})^{\mathrm{Tr}_{\Z_{q^{k}}/\Z_p}( [a_{ij}] [x]^j )}  = \prod_{l=0}^{ak-1} E(\pi_{(i,j)}( [a_{ij}] [x]^j )^{p^l}) =\prod_{l=0}^{ak-1}  E_{f,(i,j)}^{\sigma^l}([x]^{p^l}).
\end{align*}
If one takes the product over all $(i,j) \in \mathfrak{X}$, the result follows. 
\end{proof}

Let $G=\Gal(\Q_q/\Q_p)=\langle \sigma \rangle$ where $\sigma$ is the Frobenius element. The group $G$ acts on a power series in $R'$ via the coefficient ring $\Z_q$ and trivially on all $\pi_x$'s.
Its action on $R'$ extends to $B$ by acting trivially on $X$.
Let $\psi_p$ be the $R'$-linear operator 
\begin{align*}
\psi_p( \sum_{i \geq 0} c_i X^i) = \sum_{i \geq 0} c_{pi} X^i
\end{align*}
on $B$. Any $b \in B$ induces a linear operator from $B$ to $B$ by the multiplication by $b$ map, which we denote just by $b$.
Then define the Dwork operator $\psi$ on $B$ by 
$$
\psi=\sigma^{-1} \circ \psi_p \circ E_f(X),
$$
where $E_f(X)$ is the multiplication by $E_f(X)$ map.
This operator is $R$-linear, but generally not $R'$-linear. 
For $r(X), s(X) \in B$ one has 
$r(X) \circ \psi_p =\psi_p \circ r(X^p)$,
$\sigma^{-1} \circ r(X) = r^{\sigma}(X)\circ \sigma^{-1}.
$
Furthermore, $\sigma^{-1}$ and $\psi_p$ commute.
For any integer $k \in \Z_{\geq 0}$ we find
$
\psi^k=\sigma^{-k} \circ \psi_p^k \circ \left( \prod_{i=0}^{k-1} E_f^{\sigma^i}(X^{p^i}) \right).
$
In particular, we see that $\psi^a$ is $R'$-linear and satisfies
$
\psi^a= \psi_p^a \circ \left( \prod_{i=0}^{a-1} E_f^{\sigma^i}(X^{p^i}) \right).
$

Recall that $\Gamma=\{X^i: i \in \Z_{\geq 0}\}$ is a basis of $B$ as Banach $R'$-module. Set $\beta_i=0$ for $i<0$. 

\begin{lemma} \label{777}
Write 
$\prod_{i=0}^{k-1} E_f^{\sigma^i}(X^{p^i}) =\sum_{l=0}^{\infty} \beta_l X^l$ 
with $\beta_l \in R'$. Then one has $v(\beta_l)\geq \frac{l}{p^{k-1}\delta}$.
For $k \in \Z_{\geq 1}$ the map $\sigma^k \circ \psi^k$ is $R'$-linear and one has $\Mat(\sigma^k \circ \psi^k|\Gamma)=  (\beta_{p^k i-j})_{i,j} $. 
\end{lemma}
\begin{proof}
By Lemma \ref{104}, one has $v(\beta_l) \geq \frac{l}{p^{k-1} \delta}$.
On the other hand we have
\begin{align*}
\sigma^k \circ \psi^k \left(X^i \right)  = \sum_{l=0}^\infty \beta_{p^kl-i}  X^l.
\end{align*}
\end{proof}

\begin{theorem}[Dwork's trace formula] \label{1123}
For any  positive integer $k$ one has
\begin{align*}
S^*(k,\pi) = (q^k-1)\cdot  \mathrm{Tr}(\psi^{ak}|\Gamma) \in R. 
\end{align*}
\end{theorem}
\begin{proof}
Consider the linear map $\psi^{ak}$. Write
$
\prod_{i=0}^{ak-1} E_f^{\sigma^i}(X^{p^i})) = \sum_{i=0}^{\infty} \beta_i X^i.
$ From Lemma \ref{777} one obtains
$
 \mathrm{Tr}(\psi^{ak}|\Gamma) = \sum_{j=0}^{\infty} \beta_{(q^k-1)j}\in R.
$
Note that 
$$
\sum_{x \in \F_{q^k}^*} [x]^i = \left\{ 
\begin{array}{ll} q^k-1 & \textrm{if }(q^k-1)| i \\ 0 & \textrm{otherwise.} \end{array} \right.
$$
Lemma \ref{120} then gives
\begin{align*}
(q^k-1) \cdot \mathrm{Tr}(\psi^{ak}|\Gamma) = (q^k-1)  \sum_{j=0}^{\infty} \beta_{(q^k-1)j} = \sum_{x \in \F_{q^k}^*} \sum_{j=0}^{\infty} \beta_j [x]^j \\
 = \sum_{x \in \F_{q^k}^*}  \prod_{i=0}^{ak-1} E_f^{\sigma^i}([x]^{p^i}) = S^*(k,\pi).
\end{align*}
\end{proof}

\begin{corollary}[Analytic trace formula] \label{ana}
One has
\begin{align*}
C^*(\pi,s) = \det(I-s\psi^a|\Gamma).
\end{align*}
\end{corollary}
\begin{proof}
This follows from Theorem \ref{1123} and the identity
\begin{align*}
\det(I- s\psi^a|\Gamma) = \exp \left( - \sum_{k=1}^{\infty} \mathrm{Tr}(\psi^{ak}|\Gamma) \frac{s^k}{k} \right) = C^*(\pi,s) \in R[[s]].
\end{align*}
This identity holds since $\psi^a$ is a nuclear operator by Lemma \ref{777}.
\end{proof}

Finally we can prove a lower bound on the $\pi$-adic Newton polygon of $C^*(\pi,s)$. 

\begin{proposition}[Hodge lower bound] \label{hod}
The $\pi$-adic Newton polygon of $C^*(\pi,s)$ lies above the polygon whose slopes are 
\begin{align*}
0, \frac{a(p-1)}{\delta}, \frac{2a(p-1)}{\delta}, \ldots,
\end{align*}
that is, above the polygon with vertices $(k, \frac{a(p-1)k(k-1)}{2\delta})$ (for $k=0,1,2,\ldots$). 
\end{proposition}
\begin{proof}
By Corollary \ref{ana} and some standard properties of matrices one obtains 
\begin{align*}
\prod_{\zeta \in \mu_a} \mathrm{det}(I-s\psi \zeta |B/R) =& \mathrm{det}(I-s^a\psi^a| B/R)= \prod_{i=0}^{a-1} \sigma^i \left( \mathrm{det}(I-s^a\psi^a| \Gamma) \right) \\
=& C^*(\pi,s^a)^a.
\end{align*}
This implies that 
 the $\pi$-adic Newton polygon of $C^*(\pi,s^a)$ is equal to the Newton polygon of $\det(I-s\psi|B/R)$.  From the estimate in Lemma \ref{777} for $k=1$, it is a standard computation to show that the $\pi$-adic Newton polygon of
$\det(I-s\psi|B/R)$ has lower bound 
$
(ak, \frac{a(p-1)k(k-1)}{2\delta}) \textrm{ for } k=0,1,\ldots.
$
We dilate this result to obtain the lower bound for the Newton polygon of $C^*(\pi,s)$. 
\end{proof}

\section{Upper and lower bounds on $C^*(\pi,s)$ and $C^*(\chi,s)$}

In this section, we will give upper and lower bounds on the $\pi$-adic Newton polygon of $C^*(\pi,s)$ and the $\pi_{\chi}$-adic Newton polygon of $C^*(\chi,s)$. Recall that $\delta=d_m/p^m$ for some fixed $m$. 

\begin{lemma}  \label{lola} 
Let $r \in R$. For any non-trivial finite character $\chi: \Z_p \to \C_p^*$ one has 
\begin{align*}
v_{\pi_{\chi}}\left( \mathrm{ev}_{\chi}(r) \right)\geq v(r).
\end{align*}
If $v_{\pi_{\chi}}\left( \mathrm{ev}_{\chi}(r) \right)= v(r)$ for a finite non-trivial character $\chi$ with $m_{\chi}>m$, then one has $v_{\pi_{\chi'}}\left( \mathrm{ev}_{\chi}(r) \right)=v(r)$ for any finite character $\chi'$ with $m_{\chi'}>m$.
\end{lemma}
\begin{proof}
We will first show that for $r \in R$ one has $v_{\pi_\chi}(\ev_\chi(r))\ge v(r)$. By Lemma \ref{200} and Lemma \ref{540} one has
\begin{align*}
v_{\pi_\chi}(\ev_\chi(\pi_{(i,j)}))=v_{\pi_\chi}(\pi_i(\pi_{\chi}))  \geq p^i v_{\pi_{\chi}}(\pi_{\chi})= p^i =v(\pi_{(i,j)})+  \Delta_{(i,j)} \geq v(\pi_{(i,j)}).
\end{align*}
The result easily generalizes to any $r \in R$ by the definition of $v$. 

We will now prove the second statement. 
The result follows easily when $r \in \Z_p$. Recall that $\delta = d_m/p^m$. We make the following claim, from which one easily deduces the result.  Claim: For $r \in R \setminus \Z_p$ one has $v_{\pi_{\chi}}\left( \mathrm{ev}_{\chi}(r) \right) = v(r)$ if the following both hold:
\begin{itemize}
\item $m_{\chi}>m$;
\item $r= c \cdot \pi_{(m,d_m)}^{e} +(\mbox{other terms})$ with $c \in \Z_p^*$ and $e \in \Z_{\geq 0}$ such that $v(r)=v( \pi_{(m,d_m)}^{e})$. 
\end{itemize}
Recall that for $(i,j) \in \mathfrak{X}$ one has: 
\begin{align*}
 v_{\pi_{\chi}}    \left( \mathrm{ev}_{\chi}(\pi_{(i,j)}) \right)  \geq v(\pi_{(i,j)})+ \Delta_{(i,j)}.
\end{align*}
One then finds for any nonzero $r=\sum_{u\in U} c_u \pi^u$ in $R$:
\begin{align*}
 v_{\pi_{\chi}}\left( \mathrm{ev}_{\chi}(r) \right)\geq& \mathrm{min}\{ v_{\pi_{\chi}}(c_u)+v_{\pi_{\chi}}   \left( \mathrm{ev}_{\chi}(\pi^u) \right): u \in U,\ c_u \neq 0 \}\\
\geq& \mathrm{min}\{ v_{\pi_{\chi}}   \left( \mathrm{ev}_{\chi}(\pi^u) \right): u \in U,\ c_u \neq 0 \} \\
\geq& \mathrm{min}\{v(\pi^u) + \sum_{x \in \mathfrak{X}} u(x) \Delta_x:  u \in U,\ c_u \neq 0 \} \\
\geq& \mathrm{min}\{v(\pi^u): u \in U,\ c_u \neq 0\}=v(r).
\end{align*}
We will now check when equality holds. The last equality holds if and only if there is a $u'$ with $c_{u'} \neq 0$ with $v(\pi^{u'})=v(r)$ and $\pi^{u'}=\pi_{(m,d_m)}^{e}$ for some $e \in \Z_{\geq 0}$ by Lemma \ref{540} (this $e$ is unique). The third equality holds if and only if in addition one has $m_{\chi}>m$. The second also holds if and only if in addition one has $c_{u'} \in \Z_p^*$. If this is the case, then the first equality automatically holds (all other terms will have higher valuation). The result follows.
\end{proof}


Write  $C^*(\pi,s) = \sum_{k\geq 0}b_k s^k$ with $b_k\in R$.

\begin{proposition} \label{ho2}
Let $\chi: \Z_p \to \C_p^*$ be a finite character with $m_{\chi}>m$. 
The $\pi$-adic Newton polygon of $C^*(\pi,s)$ and the
$\pi_{\chi}$-adic Newton polygons of $C^*(\chi,s)$ lie above the polygon with vertices 
\begin{align*}
\left(k, \frac{a(p-1)k(k-1)}{2\delta} \right) \qquad \textrm{for }k=0,1,2,\ldots.
\end{align*}
and lie below the polygon with vertices  
\begin{align*}
\left( k, \frac{a(p-1)k(k-1)}{2\delta} \right) \qquad 
\mbox{for $k \in \Z_{\geq 0}$, $k \equiv 0, 1\bmod d_m$}.
\end{align*}
\end{proposition}
\begin{proof} 
The $\pi$-adic Newton polygon, by the Hodge bound of Proposition \ref{hod},
lies over the polygon with vertices at $(k,\frac{a(p-1)k(k-1)}{2\delta})$ for all $k\geq 0$. 
Let $\chi: \Z_p \to \C_p^*$ be a finite character with $m_{\chi}>m$. Note that $L(\chi,s) = \ev_{\chi}(L(\pi,s))$. 
The $\pi_{\chi}$-adic Newton polygon $C^*(\chi,s)$ lies above the $\pi$-adic one by Lemma \ref{lola} because for each coefficient $b_k\in R$ of $C^*(\pi,s)$ we have 
$v_{\pi_\chi}(\ev_{\chi}(b_k))\geq v(b_k)$. 
This finishes the proof of the lower bound in both cases.

Let $\chi_0$ is a character with $m_{\chi_0}=m+1$. By the Weil conjectures, 
$L(\chi_0,s)=1+\cdots + c_{d_m-1}s^{d_m-1}$ with
$v_{\pi_{\chi_0}}(c_{d_m-1})=\frac{(d_m-1)}{2} \cdot a(p-1)p^m$, since  $
v_{\pi_{\chi_0}}(q)= a \varphi(p^{m+1})=a(p-1)p^m$.
This shows that the $\pi_{\chi_0}$-adic Newton polygon of $L(\chi_0,s)$ always 
has the following two endpoint vertices $(0,0)$ and 
$
(d_m-1,\frac{(d_m-1) a(p-1)p^m }{2}).
$
Note that 
\begin{align*}
L^*(\chi_0,s)=(1-\chi_0(\mathrm{Tr}_{\Z_q/\Z_p}( f(0))s) \cdot  L(\chi_0,s), 
\end{align*}
and that $(1-\chi_0(\mathrm{Tr}_{\Z_q/\Z_p}( f(0))s)$ just contributes a segment of slope $0$.
Write $0=\alpha_1\leq \ldots\le \alpha_{d_m}<a(p-1)p^m$ for the
$\pi_{\chi_0}$-adic slopes of $L^*(\chi_0,s)$
where $\sum_{k=1}^{d_m}\alpha_k=a(p-1)p^m(d_m-1)/2$.
Then 
$
C^*(\chi_0,s)=\prod_{i\geq 0}L^*(\chi_0,q^i s)
$
has slopes 
$$
a(p-1)p^mi=a(p-1)p^mi+\alpha_1\leq \cdots \leq a(p-1)p^m i+\alpha_{d_m} 
$$
for $i=0,1,\ldots$, in strictly increasing order as $i$ increases.
This implies that for any $k\equiv 0\bmod d_m$ we have fixed vertices:
$x=k$ and 
$$y=a(p-1)p^m\frac{k(k-1)}{2d_m}=\frac{a(p-1)k(k-1)}{2\delta}.$$
That is $(x,y)=(k,\frac{a(p-1)k(k-1)}{2\delta})$. 
A similar argument yields vertices at 
$(x,y)=(k,\frac{a(p-1)k(k-1)}{2\delta})$ for $k\equiv 1\bmod d_m$.
This proves the upper bound for the $\pi_{\chi_0}$-adic Newton polygon of $L(\chi_0,s)$. By Lemma \ref{lola} one has  $v_{\pi_{\chi_0}}\left( \mathrm{ev}_{\chi} (b_k)  \right) \geq v(b_k)$, and hence the same upper bounds hold for the $\pi$-adic Newton polygon of $C^*(\pi,s)$. By the shape of the upper and lower bound one sees that the points $(k,\frac{a(p-1)k(k-1)}{2\delta})$ for $k \equiv 0,1 \bmod{d_m}$ are vertex points of the $\pi$-adic Newton polygon of $C^*(\pi,s)$ and of the $\pi_{\chi_0}$-adic Newton polygon of $C^*(\chi_0,s)$. 
Hence for such $k$ we have
\begin{align*}
v(b_k) = v_{\pi_{\chi_0}}\left( \mathrm{ev}_{\chi} (b_k) \right) = \frac{a(p-1)k(k-1)}{2\delta}.
\end{align*}
By Lemma \ref{lola} this equality holds when $\chi_0$ is replaced by any $\chi$ with $m_{\chi}>m$. This gives the desired upper bound on the $\pi_{\chi}$-adic Newton polygon of $C^*(\chi,s)$.
\end{proof}

\section{Slope uniformity} \label{sla}

We shall now translate the results about $C^*(\chi,s)$ to $L(\chi,s)$ to the proof of Theorem A in this section. 
We obtain the following theorem (we follow the formulation of \cite[Theorem 1.1]{REN}).

\begin{theorem} \label{444}
Let  $\chi: \Z_p \to \C_p^*$ be a non-trivial finite character of order $p^{m_{\chi}}$ with $m_{\chi}>m$. Then $L_f(\chi,s)$ is a polynomial of degree 
$
\delta p^{m_{\chi}-1}-1.
$
Write
\begin{align*}
L_f(\chi,s)= \sum_{i=0}^{\delta p^{m_{\chi}-1}-1} c_i s^i. 
\end{align*}
We have the following.
\begin{enumerate}
\item For any $0<n \leq p^{m_{\chi}-m-1}$, we have 
\begin{align*}
v_q(c_{nd_m-1 }) = \frac{n(n d_m-1)}{2 p^{m_{\chi}-m-1}} \textrm{ and } v_q(c_{nd_m}) = \frac{n(n d_m+1)}{2 p^{m_{\chi}-m-1}}.
\end{align*} 
\item For any $0<n \leq p^{m_{\chi}-m-1}$, the $q$-adic Newton polygon of $L_f(\chi,s)$ passes through the points 
\begin{align*}
\left(nd_m-1, \frac{n(n d_m-1)}{2 p^{m_{\chi}-m-1}} \right) \textrm{ and }\left( n d_m, \frac{n(n d_m+1)}{2 p^{m_{\chi}-m-1}} \right).
\end{align*}
\item The $q$-adic Newton polygon of $L_f(\chi,s)$ has slopes (in multiset notation and in increasing order)
\begin{align*}
\bigcup_{i=1}^{p^{m_{\chi}-m-1}} \{ \alpha_{i1}, \alpha_{i2}, \ldots , \alpha_{id} \}-\{0\},
\end{align*}
where
\begin{align*}
\left\{  \begin{array}{cc} \alpha_{ij} = \frac{i-1}{p^{m_{\chi}-m-1}} & \textrm{ if } j=1 \\
\frac{i-1}{p^{m_{\chi}-m-1}} <  \alpha_{ij} < \frac{i}{p^{m_{\chi}-m-1}} & \textrm{ if }j>1. \end{array} \right.
\end{align*}
\end{enumerate}
\end{theorem}
\begin{proof}
Let $u_{1} \leq u_2 \leq \ldots \leq u_{d_m p^{m_{\chi}-m-1}-1}$ be the $q$-adic slopes of $L(\chi,s)$. The $q$-adic slopes of $L^*(\chi,s)$ are $0, u_1, u_2,\ldots, u_{d_m p^{m_{\chi}-m-1}-1}$. Hence the $q$-adic slopes of $C^*(\chi,s)$ are (in multiset notation)
\begin{align*}
\bigcup_{i \geq 0} \{ i, u_1+i, u_2+i,\ldots, u_{d_m p^{m_{\chi}-m-1}-1}+i \}.
\end{align*}
If we compare these with the obtained upper and lower bound in Proposition \ref{ho2} and one normalizes properly, one obtains the result. 
\end{proof}

\begin{proof}[{\bf Proof of Theorem A}]
The valuations $v_q(\alpha_{\chi,i})$ correspond to the slopes of the $q$-adic Newton polygon of $L(\chi,s)$ by basic properties of Newton polygons. The result then follows easily from Theorem \ref{444}. 
\end{proof}

\section{Slope stability} \label{slo}

In this section, we will show that many genus stable covers are slope stable.
Our strategy is as follows. First we show that if $f_i=0$ for large enough $i$, that the cover is slope stable. Then we show that in many cases we can reduce to the case $f_i=0$ for large enough $i$, since many coefficients of the $f_i$ do not influence the Newton polygon of $L(\chi,s)$.

Recall from the introduction that we call the $\Z_p$-cover $\mathcal{P}$ $m'$-slope stable with $m' \in \Z_{\geq 1}$ if the following holds. Let $\chi_0$ be a character with $m_{\chi_0}=m'$. Let $u_1,\ldots, u_{\mathrm{deg}(L(\chi_0,s))}$ denote the slopes of the $q$-adic Newton polygon of $L(\chi_0,s)$. Then for every character $\chi$ with $m_{\chi} \geq m'$ the $q$-adic Newton polygon of $L(\chi,s)$ has slopes (in multiset notation)
\begin{align*}
\bigcup_{i=0}^{p^{m_{\chi}-m'}-1} \left\{  \frac{i}{p^{m_{\chi}-m'}}, \frac{u_1+i}{p^{m_{\chi}-m'}}, \ldots, \frac{u_{\mathrm{deg}(L(\chi_0,s))}+i}{p^{m_{\chi}-m'} }  \right\} - \{0\}.
\end{align*}

\begin{lemma} \label{yt}
The tower $\mathcal{P}$ is $m'$-slope stable if and only if $\mathcal{P}$ is genus stable with 
$\delta=\frac{d_m}{p^m}$ for some $m<m'$ and the $\pi_{\chi}$-adic Newton polygon $C^*(\chi,s)$ does not depend on $\chi$ when $m_{\chi} \geq m'$.  
\end{lemma}
\begin{proof}
Let $\chi_0$ be a character with $m_{\chi_0}=m'$. Let $u_1,\ldots, u_{\mathrm{deg}(L(\chi_0,s))}$ denote the slopes of the $q$-adic Newton
polygon of $L(\chi_0,s)$. Set $t=a (p-1) p^{m'-1 }$.

$\implies$: Assume that $\mathcal{P}$ is $m'$-slope stable. Then $\mathcal{P}$ is genus stable 
with $\delta=\frac{d_m}{p^m}$ for some $m < m'$.
 Furthermore, for $\chi$ a character with $m_{\chi} \geq m'$, since $v_{\pi_{\chi}}(q)=a (p-1) p^{m_{\chi}-1}$,
 the polynomial $L^*(\chi,s)$ has $q$-adic slopes
\begin{align*}
\bigcup_{i=0}^{p^{m_{\chi}-m'}-1} \left\{  \frac{i}{p^{m_{\chi}-m'}}, \frac{u_1+i}{p^{m_{\chi}-m'}}, \ldots, \frac{u_{\mathrm{deg}(L(\chi_0,s))}+i}{p^{m_{\chi}-m'} }  \right\}.
\end{align*} 
It follows $C^*(\chi,s)=\prod_{i=0}^{\infty} L^*(\chi,q^is)$  has $\pi_{\chi}$-adic slopes
\begin{align*}
\bigcup_{j=0}^{\infty} \bigcup_{i=0}^{p^{m_{\chi}-m'}-1} 	\left\{  (i+j p^{m_{\chi}-m'})t , (u_1+i+j p^{m_{\chi}-m'}) ) t, \ldots,   (u_{\mathrm{deg}(L(\chi_0,s))}+i+j p^{m_{\chi}-m'}) t  \right\} \\
= \bigcup_{i=0}^{\infty}\{ it, (u_1+i)t, \ldots, (u_{\mathrm{deg}(L(\chi_0,s))}+i)t \}
\end{align*}
The latter formula does not depend on $\chi$ and the implication follows.

$\limplies$: The proof is similar to that of 
\cite[Theorem 1.2]{WAN7}. Let $\chi$ be a finite character with  $m_{\chi}  \geq m'$. Let $w_1,\ldots,w_{ d_m p^{m_{\chi}-m}} \leq 1$ be the $q$-adic slopes of $L^*(\chi,s)$. Then the $\pi_{\chi}$-adic Newton polygon of $C^*(\chi,s)$ 
has slopes
\begin{align*}
\bigcup_{i=0}^{\infty}\{ (w_1+i)t p^{m_{\chi}-m'}, (w_2+i)t p^{m_{\chi}-m'}, \ldots, (w_{ d_m p^{m_{\chi}-m}}+i)t p^{m_{\chi}-m'} \}.
\end{align*}
As these slopes are independent of the characters, 
this gives slopes for $\chi=\chi_0$, one hence obtains the slope stability. 
\end{proof}

Given a genus stable $\Z_p$-cover $\mathcal{P}$ with $\delta$, set 
\begin{align*}
W=a(p-1) \frac{(d_m-1)^2}{8 \delta}.
\end{align*}

\begin{lemma} \label{bobo}
The upper and lower bounds from Proposition \ref{ho2} for the $\pi$-adic Newton polygon of $C^*(\pi,s)$ differ by at most $W$. Also, for a finite character $\chi: \Z_p \to \C_p^*$ with $m_{\chi}>m$, the upper and lower bounds for the $\pi_{\chi}$-adic Newton polygon of $C^*(\chi,s)$ from Proposition \ref{ho2} differ by at most $W$.
\end{lemma}
\begin{proof}
We omit the proof here as it is almost identical to that of \cite[Lemma 3.7]{WAN7}. 
\end{proof}

\begin{proposition} \label{om}
Let $N\in\Z_{\geq 1}$. Suppose a genus stable
$\Z_p$-cover $\mathcal{P}$ has 
$f_i=0$ for all $i>N$,
then $\mathcal{P}$ is $m'$-slope stable
for  
$
m' = 1 +  \lceil \log_p( \frac{p^N+W}{p-1}) \rceil.
$
\end{proposition}

\begin{proof}
 If $W=0$ then $d_m=1$ and $m=0$ (the converse is also true).  
The upper and lower bounds in Lemma \ref{bobo}
then completely coincide, so $\mathcal{P}$ is $1$-slope stable. For the rest of the proof we assume $W>0$.

Let $\chi: \Z_p \to \C_p^*$ be a character with $m_{\chi} \geq m'>m$.  
We have 
$
\pi_i(T) = T^{p^i} h_i + p g_i
$
with $h_i \in 1 + T\Z_p[[T]]$ and $g_i \in \Z_p[[T]]$ by Lemma \ref{summer}. 
For $u \in U$, as defined in Section \ref{543}, we set $w(u)=\sum_{(i,j) \in \mathfrak{X}} p^i \cdot u((i,j))
\geq v(\pi^u)$. 
By the choice of $m'$ one has for $i=0,1,\ldots,N$
\begin{align*}
v_{\pi_{\chi}}(p)= \varphi(p^{m_{\chi}}) \geq  p^N+W \geq p^i + W.  
\end{align*}
Hence for $u \in U$ with $u((i,j))=0$ if $i>N$ one finds
\begin{align*} \label{913}
\mathrm{ev}_{\chi} (\pi^u) = \pi_{\chi}^{w(u)}h_u(\pi_{\chi})+ \pi_{\chi}^{w(u)+\pceil{W}} z_u
\end{align*}
where $h_u \in \Z_p[[T]]$ does not depend on $\chi$ and $z_u \in \Z_p[\pi_{\chi}]$.  
For $i \in \Z_{\geq 0}$ set $t_i=\lceil \frac{a(p-1)i(i-1)}{2\delta} \rceil$ and write $C^*(\pi,s)=\sum_{i\ge 0} b_i s^i$ where
\begin{align*}
b_i = \sum_{u \in U:\  w(u) \geq t_i} s_u \pi^u
\end{align*}
with $s_u \in \Z_p$ (lower bound in Proposition \ref{ho2}). We then find
\begin{align*}
\mathrm{ev}_{\chi} (b_i )  =& \sum_{u \in U:\  w(u) \geq t_i} s_u \left(  \pi_{\chi}^{w(u)}h_u(\pi_{\chi})+ \pi_{\chi}^{w(u)+\pceil{W}} z_u \right) =  
\sum_{j\geq t_i}g_{i,j}\pi_\chi^j
\end{align*}
for some $g_{i,j}\in\Z_p$. Notice that 
$g_{i,j}$ is independent of $\chi$ for all $t_i\leq j <  t_i+W$.
Let $j^*$ be the minimal of all $j<t_i+W$ such that $g_{i,j}\in\Z_p^*$. Assume first that $j^*$ exists. Notice that $j^*$ does not depend on $\chi$.
For $t_i\leq j<j^*$ we have $v_{\pi_\chi}(g_{i,j}\pi_\chi^j)\geq t_i + v_{\pi_\chi}(p)>t_i+W>j^*$ by hypothesis and we find  $v_{\pi_\chi}(\ev_{\chi}(b_i)) = j^*$. If such $j^*$ does not exist,
we have $v_{\pi_\chi}(\ev_\chi(b_i))\geq t_i+ v_{\pi_\chi}(p)>t_i+W$, which lies above our upper bound of the Newton polygon in Proposition \ref{ho2} and Lemma \ref{bobo}. Hence such $b_i$'s do not affect our Newton polygon. This shows that the Newton polygon of $C^*(\chi,s)$ is independent of $\chi$.
\end{proof}

We will now study how much the Newton polygon of $C^*(\chi,s)$ depends on the coefficients $a_{ij}$ that defines the $\Z_p$-cover $\mathcal{P}$. 

\begin{theorem} \label{om2}
Let $\mathcal{P}$ and $\mathcal{P}'$ be two genus stable $\Z_p$-covers of $\Ps^1_k$ with $\delta=\delta(\mathcal{P})=\delta(\mathcal{P}')=\frac{d_m}{p^m}$ for some $m \in \Z_{\geq 0}$, defined by $(a_{ij})_{i,j}$ and $(a'_{ij})_{i,j}$ (respectively) as in Section \ref{ha4}. Let $\chi: \Z_p\to \C_p^*$ be a finite character with $m_{\chi} >m$. 
If $a_{ij}=a'_{ij}$ for all $(i,j) \in \{ x \in \mathfrak{X}: \Delta_x< W \}$,
then the $\pi_\chi$-adic Newton polygons of $C^*(\chi,s)$ for these two covers are identical. 
\end{theorem}
\begin{proof}
The upper and lower bounds of the $\pi$-adic Newton polygon of $C^*(\pi,s)$ and the $\pi_{\chi}$-adic Newton polygon of $C^*(\chi,s)$ differ by at most $W$ by Lemma \ref{bobo}. If $(i,j) \in \mathfrak{X}$ with $\Delta_{(i,j)} \geq W$, one has
\begin{align*}
v_{\pi_\chi}(\mathrm{ev}_{\chi} ( \pi_{(i,j)} )) = v_{\pi_{\chi}}(\pi_i(\pi_{\chi})) \geq p^i= v(\pi_{(i,j)})+ \Delta_{(i,j)} \geq v(\pi_{(i,j)}) + W.
\end{align*}
Hence terms in $C^*(\pi,s)$ involving such $\pi_{(i,j)}$ contribute to points above the upper bound of the $\pi_{\chi}$-adic Newton polygon of $C^*(\chi,s)$ (Proposition \ref{ho2}, Lemma \ref{bobo}). Hence the corresponding coefficient $a_{ij}$ giving rise to $\pi_{(i,j)}$ are irrelevant.
\end{proof}

\begin{theorem} \label{jolo} 
Assume $W>0$. If 
\begin{align*}
N=\max \{ i: \exists (i,j) \in \mathfrak{X} \textrm{ with } a_{ij} \neq 0 \textrm{ and } \Delta_{(i,j)}<W  \}
\end{align*}
exists,  then $\mathcal{P}$ is $m'$-slope stable  for 
$m'=1 + \lceil \log_p (\frac{p^N+W}{p-1})\rceil >m$.
\end{theorem}
\begin{proof}
Let $\chi: \Z_p \to \C_p^*$ be a character with  $m_{\chi} >m'$. Then by Theorem \ref{om2}, the $\pi_{\chi}$-adic Newton polygon of $C^*(\pi_{\chi},s)$ is equal to the one where we replace $a_{ij}$ by zero for $i>N$. Then we apply Proposition \ref{om}. 
\end{proof}

\begin{proof}[{\bf{Proof of Theorem B}}]
Set $d_i=\deg(f_i)= \max \left( \{j: a_{ij} \neq 0\} \cup \{-\infty\} \right)$. The condition in Theorem \ref{jolo} is equivalent to the following. There exists an integer $N$ such that for $i>N$ one has $p^i-\frac{d_i}{\delta} \geq W$, that is, 
\begin{align*}
d_i \leq (p^i-W) \delta = \delta p^i  - W \delta.
\end{align*}
The result follows.
\end{proof}

We want to remark that 
Theorem B generalizes the main theorem of \cite[Theorem 1.2]{WAN7}:
Assume that $f_i=0$ for $i \geq 1$. Then $\mathcal{P}$ is slope stable. In fact it also generalizes the following theorem of Li 
(see \cite{LIX}):
Assume that $\{\deg(f_i): i \in \Z_{\geq 0} \}$ is bounded. Then $\mathcal{P}$ is slope stable. 

\begin{remark} \label{consta}
The proof of Theorem B shows that in Theorem B one can take
\begin{align*}
C= \frac{a(p-1)(d_m-1)^2}{8}.
\end{align*}
\end{remark}

\begin{example} \label{522}
It is easy to construct genus stable $\Z_p$-towers for which Theorem B does not apply. For example take an integer $b>1$ with $p \nmid b$ and consider the tower defined by  
\begin{align*}
f= [X]^b + \sum_{i=1}^{\infty} p^i [X]^{bp^i-1}.
\end{align*}
One has $m=0$ and $\delta=b$. Assume that $W> \frac{1}{b}$, which can be achieved by taking $b$ large enough. For $i \geq 1$ one finds
\begin{align*}
\deg(f_i) = bp^i-1 = \delta p^i - \frac{\delta}{b} > \delta p^i - W \delta.
\end{align*}
Hence Theorem B is not able to tell if such cover is slope stable.
\end{example}

\begin{remark} \label{mok}
Let $\chi: \Z_p\to \C_p^*$ be a continuous character of infinite order. Such a character is completely determined by $\chi(1) \in \C_p^*$. Set $\pi_{\chi}=\chi(1)-1$ and assume that $v_p(\pi_{\chi})>0$. What can one say about the $\pi_{\chi}$-adic Newton polygon of  $\mathrm{ev}_{\chi} (C^*(\pi,s))$? 

It turns out that Lemma \ref{200} does not hold for $x \in \C_p$ with $v_p(x)>0$ and $x+1 \not \in \mu_{p^{\infty}}$. For such $x$ one can show that there are $C(x) \in \Q_{\geq 0}$ and $M(x) \in \Z_{\geq 0}$ such that
\begin{align*}
v_p(\pi_i(x))= v_p((x+1)^{p^i}-1) = 
\left\{  \begin{array}{ll} 
p^i v_p(x)  & \textrm{if }i< M(x) \\ 
C(x)+i & \textrm{if }i \ge M(x). \end{array} \right.
\end{align*}	
The main reason for this is the following. One has
$
(x+1)^{p}-1=\sum_{j=1}^{p-1}\binom{p}{j}x^j+x^{p}.
$
If the valuation of $x$ is small, then the valuation of $(x+1)^{p}-1$ is determined by the valuation of $x^p$. If the valuation of $x$ is large, the valuation is determined by the valuation of the term $px$. See \cite{KO16} for a more thorough treatment of the
$p$-th power maps. 

To $\pi_{\chi}$ we can associate $C$ and $M$ as above. The lower bound as in Proposition \ref{ho2} holds if for all $(i,j) \in \mathfrak{X}$ with $a_{ij} \neq 0$ one has 
\begin{align*}
 v_{\pi_{\chi}}( \mathrm{ev}_{\chi}(\pi_{(i,j)}))   = v_{\pi_{\chi}}(\pi_i(\pi_{\chi})) \geq v(\pi_{(i,j)}).
\end{align*}
This condition is automatically satisfied if $i<M$. For $i \geq M$, if $a_{ij}=0$, one needs 
\begin{align*}
(C+i) v_{\pi_{\chi}}(p) \geq v(\pi_{(i,j)})= \frac{j}{\delta}.
\end{align*}
In other words, for $i \geq M$ one needs 
\begin{align*}
\deg(f_i) \leq \delta (C+i) v_{\pi_{\chi}}(p).
\end{align*}
For the upper bound as in Proposition \ref{ho2} to hold, in addition we need
\begin{align*}
v_{\pi_{\chi}}( \mathrm{ev}_{\chi}(\pi_{(m,d_m)}) =v_{\pi_{\chi}}( \pi_{m} ( \pi_{\chi})) = p^m = v_{\pi}(\pi_{(m,d_m)}).
\end{align*}
The latter is equivalent to $m \leq M$. Hence both upper and lower bounds seem to hold for much smaller classes of towers. This gives us reason to believe that purely with $L^*(T,s)$ one can prove Theorem B only when $\deg(f_i)$ is at most linear in $i$. 
\end{remark}

\end{document}